\documentclass[11pt]{amsart}
 \usepackage{amssymb,amsmath,amscd}

\usepackage{pgf,tikz, subfigure}
\usepackage{mathrsfs}
\usepackage{epstopdf}
\usetikzlibrary{arrows}
\newtheorem{theorem}{Theorem}[section]

\newtheorem{lemma}[theorem]{Lemma}
\newtheorem{corollary}[theorem]{Corollary}
\newtheorem*{theorem*}{Theorem}
\theoremstyle{definition}
\newtheorem{defi}[theorem]{Definition}
\newtheorem{remark}[theorem]{Remark}

\newcommand{\pp}[2]{\frac{\partial#1}{\partial#2}}

\newcommand{\we}{\wedge}

\newcommand{\F}{{\bf F}}

\newcommand{\R}{\mathbb{R}}
\newcommand{\Z}{\mathbb{Z}}

\newcommand{\T}{\mathbb{T}}
\newcommand{\diff}{{\rm d }}
\newcommand{\FF}{\mathcal F}

\begin{document}
\title{Integrable systems and  closed one forms}
\author{Robert Cardona}\address{ Robert Cardona,
Laboratory of Geometry and Dynamical Systems, Department of Mathematics, Universitat Polit\`{e}cnica de Catalunya, Barcelona  \it{e-mail: robert.cardona@estudiant.upc.edu }
 }
\author{Eva Miranda}\address{ Eva Miranda,
Laboratory of Geometry and Dynamical Systems, Departament of Mathematics, Universitat Polit\`{e}cnica de Catalunya BGSMath Barcelona Graduate School of
Mathematics in Barcelona and CEREMADE (Universit\'{e} de Paris Dauphine), IMCCE (Observatoire de Paris) and IMJ (Universit\'{e} de Paris Diderot), Postal Address: Observatoire de Paris 77 Avenue Denfert Rochereau, Paris, France  \it{e-mail: eva.miranda@upc.edu, Eva.Miranda@obspm.fr}
 }\thanks{{ {Robert} Cardona is supported by a Beca de Colaboraci\'{o}n MEC. {Eva} Miranda  is supported by the Catalan Institution for Research and Advanced Studies via an ICREA Academia Prize 2016, a Chaire d'Excellence de la Fondation Sciences Math\'{e}matiques de Paris and partially supported  by the grants reference number MTM2015-69135-P (MINECO/FEDER) and reference number {2017SGR932} (AGAUR).This work is supported by a public grant overseen by the French National Research Agency (ANR) as part of the \emph{\lq\lq Investissements d'Avenir"} program (reference: ANR-10-LABX-0098).}}

\begin{abstract}In the first part of this paper we  revisit a classical topological theorem by Tischler \cite{TS} and deduce a topological result about compact manifolds admitting a set of independent closed forms proving that the manifold is a fibration over a torus. As an application we reprove the Liouville theorem for integrable systems asserting that the invariant sets or compact connected fibers of a regular integrable system is a torus. We give a new proof of this theorem (including the non-commutative version) for symplectic and more generally Poisson manifolds.
\end{abstract}

\maketitle

\section{Introduction}

A Liouville integrable system on a symplectic manifold is given by $n$ functions $f_i$ (constants of motion) which pairwise  commute $\{f_i, f_j\}=0$  with respect to the Poisson bracket defined as  $ \{f, g\}=\omega(X_f, X_g)$ where $X_f$ and $X_g$ denote the Hamiltonian vector fields associated to the smooth  functions $f$ and $g$. The set of first integrals $F=(f_1,\dots, f_n)$ is often referred to as \emph{moment map}. This notion is related to classical integration of the equations of motion and can be generalized to other geometrical settings such as that of Poisson manifolds but also to non-geometrical ones like non-Hamiltonian integrable systems.

Liouville-Mineur-Arnold  theorem on integrable systems asserts that a neighborhood of a compact invariant subset (Liouville torus) of an integrable system on a symplectic manifold $(M^{2n}, \omega)$ is fibred by other Liouville tori. Furthermore,  the symplectic form can be described as the Liouville symplectic structure on $T^*(\mathbb T^n)$ in adapted coordinates to the  fibration which can be described using the cotangent lift of translations of the base torus. In particular in adapted coordinates (action-angle) the moment map is indeed a moment map of a Hamiltonian toric action. This theorem admits generalizations to the Poisson setting \cite{LMV, kms}.

The fact that the fibers of the moment map are tori is a key point in the theory and it is a purely topological result.  This topological result often attributed to Liouville \cite{liouville} was indeed probably first observed by Einstein \cite{einstein}\footnote{We thank Alain Albouy for pointing this out.}. Probably, the best well-known proof of this fact  (see for instance \cite{duistermaat}) uses the existence of a toric action associated to the joint flow of the distribution of the Hamiltonian vector fields of the first integrals and the identification of the Liouville tori as orbits of this action. The classical proof is rich because it describes not only the Liouville torus but closeby tori but somehow diverts from the topological nature of Liouville fibers.

   Given an integrable system the set of $n$ $1$-forms associated to the first integrals $df_i$ defines a set of closed $1$-forms. These closed $1$-forms are constant on the fibers of the associated \emph{moment map}.
 In this paper we pay attention to the following fact, the regular fibers of an integrable system are naturally endowed with $n$ independent closed $1$-forms defined using symplectic duality from the the constants of motion.
 This fact together with a generalization of a result of Tischler for compact manifolds admitting $k$ closed forms yields a new proof of the classical Liouville theorem.
This new topological proof of Liouville theorem works in the Poisson and non-commutative setting as well.
 In this paper, we concentrate on topological aspects of foliations associated to a set of $k$-closed one forms in the framework of integrable systems. We suspect that a detailed dual proof of the action-angle theorem by Liouville-Mineur-Arnold theorem may be obtained as a consequence of our dual viewpoint\footnote{We thank Alain Chenciner for inspiring discussions in this direction.}. We plan to apply these techniques to generalization of cosymplectic manifolds \cite{EM} in the future.

\textbf{Organization of this paper:} In Section 1 we revisit a theorem by Tischler concerning manifolds endowed with non-vanishing closed one-forms and prove that a manifold admitting $k$ independent one-forms fibers over a torus. We apply these results in Section 2 to reprove Liouville theorem for integrable systems in the symplectic and Poisson settings for commutative and non-commutative integrable systems.

\textbf{Aknowledgements:} This work was initiated  as part of the undergraduate memoir of Bachelor in Maths of { Robert Cardona} under the supervision of { Eva Miranda} and during a research stay of { Robert Cardona} under a \emph{ Severo Ochoa: Introduction to research 2017 program} at ICMAT mentored by Daniel Peralta-Salas.  We are indebted to Daniel Peralta Salas for bringing up the connection to Tischler's theorem and extremely thankful to Alain Chenciner for his encouragement and excitement about this proof. Many thanks to Camille Laurent-Gengoux for discussions about this paper. The research in this paper is supported by the Spanish Ministry of Education under a Beca de Colaboraci\'on en Departamentos. We are particularly thankful to the \emph{Fondation Sciences Math\'ematiques de Paris} for financing the trip of Robert Cardona  to Paris during which part of this paper was written. { We are grateful to the referee for their remarks and corrections.}

\section{A topological result for manifolds with closed forms}

In this first section we prove a generalization of a result of Tichler on manifolds with closed forms. We start recalling the following theorem by Tichler:

\begin{theorem}[Tischler theorem]

Let $M^n$ be a compact manifold admitting a nowhere vanishing closed $1$-form  $\omega$, then $M^n$ is a fibration over $S^1$.

\end{theorem}

In this first section we prove the following generalization of Tichler's theorem which is stated without proof for foliations without holonomy in his paper \cite{TS}:
\begin{theorem}\label{thm:main1} Let ${M^n}$ be a compact connected  manifold admitting $k$  closed $1$-forms $\beta_i, i=1,\dots, k$ 
which are linearly independent at every point of the manifold, 
 then
	 ${M^n}$ fibers over a torus $\mathbb T^k$.
\end{theorem}

 We will need the following lemma:

\begin{lemma}[Ehresmann lemma \cite{EH}]\label{lem:ehresmann}
A smooth mapping $f:M^m\longrightarrow N^n$ between smooth manifolds $M^m$ and $N^n$ such that:
\begin{enumerate}
	\item $f$ is a surjective submersion, and
	\item $f$ is a proper map
\end{enumerate}
is a locally trivial fibration.
\end{lemma}

\begin{proof}(of Theorem \ref{thm:main1})We start by proving that the cohomology classes in $H^1(M^n,\mathbb{R})$, $\{[\beta_i]\}_{i=1}^k$ are all different. Assume the opposite $\beta_i$ and $\beta_j$ with $i\neq j$ such that $[\beta_i]=[\beta_j]$. Then there exists $f \in C^{\infty}(M^n)$ such that
{\begin{equation}\label{eqn:beta}\beta_i = \beta_j + df. \end{equation}
Since the $1$-forms $\beta_i$ are linearly independent the $k$-form $\beta_1\wedge...\wedge \beta_k$ is nowhere vanishing.  Using equation \ref{eqn:beta}  we obtain
\begin{equation}
 \beta_i \wedge \beta_j = \beta_i \wedge (\beta_i +df)= \beta_i \wedge df.
\end{equation}
But note that due to Weierstrass theorem  $f$ has a maximum and a minimum on a compact manifold, thus  $\beta_i \wedge \beta_j$ vanishes at these points (where $df=0$).} This contradicts the fact that $\beta_1 \wedge ... \wedge \beta_k$ is nowhere vanishing.

\vspace{2mm}

{Denote $p$ the first Betti number of $M^n$ and $\theta$ the usual angular coordinate in $S^1$. It is well known that there exist $p$ maps $g_j:M^n\rightarrow S^1$ such that the set of $1$-forms $g_j^*(d\theta)$ define a set of cohomology classes $[g_j^*(d\theta)]$ which is a basis of $H^1_{DR}(M^n,\mathbb{R})$.} With this basis, we can express $\beta_i$  as:
$$ \beta_i = \sum_{j=1}^{p} {a_{ij} \nu_j } + dF_i, {\text{ for }}i=1,...,k.$$
Using the argument on Tischler theorem proof \cite{TS}, we can choose appropriate $q_{ij} \in \mathbb{Q} \enspace \forall i,j,$ such that $\tilde{\beta_i}= \sum_{j=1}^p {q_{ij} \nu_j } + dF_i $  {are}  still non-singular and independent. Taking suitable $N_i \in \mathbb{Z}$ we obtain forms $\beta_i'= N_i\tilde{\beta_i}$ such that
$$ \beta_i' = \sum_{j=1}^p { k_{ij} \nu_j } + dH_i, $$
where $k_{ij}=N_iq_{ij} \in \mathbb{Z}$ and $H_i=N_iF_i \in C^{\infty}({M})$. Of course, they are also non singular and independent.

Without loss of generality we can assume $dH_i=0$. Indeed, the image $H_i \in C^{\infty}(M^n)$ is contained in a closed interval because $M^n$ is compact. Functions $H_i$ quotients to $S^1$ with a projection $\pi$, and we can redefine $g_i:=g_i+\pi\circ H_i$ for $i=1,...,k$.

Recall that the basis $\nu_j$ is defined as $\nu_j= g_j^*(d\theta)=d(\tilde{g_j})$, with $\tilde{g_j}= \theta\circ g_j$.
Hence the forms $\beta_i'$  {can be} written
$$ \beta_i'= d(\sum_{j=1}^p {p_{ij} \tilde{g_j}} ).$$
If we define the functions $\theta_i= \sum_{j=1}^p {p_{ij} \tilde{g_j}}$, { then the induced mappings on the quotient $\tilde{\theta_i}:M^n \longrightarrow S^1$} are $k$ submersions of $M^n$ to $S^1$. Consider
\begin{align*}
 \Theta : M^n &\longrightarrow S^1 \times...\times S^1= \mathbb{T}^k \\
 		p &\longmapsto (\tilde{\theta_1}(p),...,\tilde{\theta_k}(p)).
\end{align*}
  Since {the forms} $\beta_i'$ are independent in $H^1(M^n,\mathbb{R})$ this implies  $d\theta_i$ are independent seen as one-forms from $M^n$ to ${\mathbb{R}^k}$ and so $d\tilde{\theta_i}$ are also independent into ${\mathbb{T}^k}$, this implies that  $\Theta$ is a surjective submersion. Since $M^n$ is compact, we can apply Ehresmann lemma {(Lemma \ref{lem:ehresmann})} and $\Theta$ defines a locally trivial fibration.
\end{proof}
When $k=n$ we obtain the following as a corollary:

\begin{corollary}\label{thm:main} Let $M^n$ be a compact connected  manifold admitting $n$ closed $1$-forms $\beta_i, i=1,\dots, n$ which are 
which are linearly independent at every point of the manifold, then 
	 $M^n$ is diffeomorphic to a torus $\mathbb T^n$.
\end{corollary}

\begin{proof} Applying Theorem 2.3, $M^n$ fibers over a torus $\mathbb T^n$. From the invariance of domain theorem it is an immersion  because the target space  is $n$-dimensional too. Thus $\Theta$ defines a covering map but since $M^n$ is connected it defines a  diffeomorphism
 \[  M^n \cong \mathbb T^n. \]	
 \end{proof}
\section{Applications to regular integrable systems}

One of the best well-known theorem of integrable systems is \emph{Liouville-Mineur-Arnold theorem} which roughly speaking asserts that the fibers of the map $F$ defined by the first integrals describe a fibration by tori (if the ambient manifold is compact and the fibers are regular) and also that there exists {privileged} coordinates (called action-angle coordinates) in which the symplectic form can be expressed in a unique Darboux chart in a neighborhood of one of these tori.

The first statement of Liouville-Mineur-Arnold theorem says that a compact connected regular set of an integrable system is in fact a torus of dimension $n$. This theorem has been attributed to Liouville for a long time but it was indeed probably first proved by Einstein \cite{einstein}. The theorem remains valid when we consider an integrable system on a Poisson manifold but also for the so-called non-commutative systems.

In this section we apply {Corollary} \ref{thm:main} to reprove that the fibers are tori for integrable systems on symplectic and Poisson manifolds. The tools used for this new proof differ from the classical tools where a torus action is used \cite{duistermaat, LMV}.

\subsection{Liouville tori of integrable systems on symplectic manifolds}

Recall the definition of an integrable system as well as the Liouville-Mineur-Arnold theorem.

\begin{defi} An \textbf{integrable system} on a symplectic manifold $(M^{2n},\omega)$ is a set of $n$ functions $f_1,...,f_n$ generically functionally independent (i.e. $df_1\wedge...\wedge df_n \neq 0$ on a dense set) and pairwise commuting with respect to the Poisson bracket $\{f_i, f_j\}=\omega(X_{f_i},X_{f_j})=0, \forall i,j$.
\end{defi}
A point $p$ is called regular point for the integrable system if $df_1\wedge ... \wedge df_n (p)\neq 0$.
\begin{theorem}
  Let $(M^{2n},\omega)$ be a symplectic manifold and $F=(f_1,...,f_n)$ an integrable system. Let $p$ be a regular point denote $F(p)=c$ and assume  $L_c=F^{-1}(c)$ is compact and connected, then
  \begin{enumerate}
  \item {$L_c \cong \mathbb{T}^n$},
  \item A neighborhood $U$ of the torus $L_c$ is the direct product of {$\mathbb{T}^n$} and the disc $D^n$, and the fibration given by ${F}$ coincides with the projection on the disc.
  \item in a neighborhood of $L_c$, $U(L_c)$, there exist coordinates  $(\theta_1,...,\theta_n,p_1,...,p_n)$ such that $\omega$ is written $\omega= \sum_{i=1}^{n}{dp_i\wedge d\theta_i}$ and $F$ only depends of $p_1,...,p_n$.
  \end{enumerate}
  \end{theorem}

Denote by $L^n$ any connected component of $F^{-1}(c)$ (or all of it if assumed connected) and we also assume it is compact. Denote by $X_i$ the Hamiltonian vector associated to $f_i$. Observe that
 \begin{align*}
  	0& = \{ f_i, f_j \}  \\
  	& = \omega( X_i, X_j) \\
  					&= \iota_{X_i}\omega (X_j) \\
  					&= -df_i(X_j)= -X_j(f_i) \enspace \forall i,j=1,...,n. 				
 \end{align*}
and  the vector fields $X_1,...,X_n$ are tangent to $L^n$ for all $p\in L^n$. So we can indeed write   $T(L^n)_p=\langle X_{f_1},...,X_{f_n} \rangle_p $. Take now in $\mathbb{R}^n$ the canonical basis of vector fields $ \{ \partial_i=\frac{\partial}{\partial x_i} \}_{i=1}^n$ on $\mathbb R^n$ and consider their pullbacks by $F$, $S_i:=F^*(\partial_i)$, which are vector fields {on} $M^{2n}$ { which are transverse} to $L^n$. They satisfy:
 $$ S_i(f_j)=\delta_{ij}. $$  They are determined by this condition modulo $T_pL^n$.
\begin{lemma}
Let $j:L^n \longrightarrow M^{2n}$ be the inclusion of the regular level set $L^n$ into {$M^{2n}$}. Define the one-forms $\alpha_i=\iota_{S_i}\omega$. Then the one-forms $\beta_i=j^*\alpha_i$ are closed.
\end{lemma}

\begin{proof} By definition of $S_i$, we have $S_i(f_j)=\delta_{ij}$. Applying it $\forall i,j$:
\begin{align*}
\alpha_i(X_j) &= \omega(S_i,X_j) \\
			  &= -\omega(X_j,S_i) \\
			  &= - \iota_{X_j}\omega(S_i) \\
			  &= df_j(S_i)= S_i(f_j)= \delta_{ij}.
\end{align*}

To prove that $\beta_i$ is closed, we just have to check that $d\alpha_i (X_i,X_j)=0 \text{ for all } X_i, X_j  \text{ {sections of} } {TL^n}$.
\begin{align*}
d\alpha_i (X_j,X_k)&=  X_k(\alpha_i(X_j)) - X_j(\alpha_i(X_k)) - \alpha_i([X_j,{X_k}]) \\
				   &= X_k( \delta_{ij}) - X_j (\delta_{ik}) - \alpha_i(0)= 0.
\end{align*}
We conclude that $d(j^*\alpha_i)=d\beta_i=0$ and so our forms $\beta_i$ are closed in $L^n$.
\end{proof}

\begin{lemma} The $1$-forms $\beta_1,...,\beta_n$ are linearly independent at all points of $L^n$.
\end{lemma}
\begin{proof}
 As seen in the previous lemma, we have that $\beta_i(X_j)=\delta_{ij}$. We deduce that $\beta_i= {X_i}^*$, by definition of dual basis. Since $X_1,...,X_n$ form a  basis of the tangent space at every point in $L^n$, we have that $\beta_1,...,\beta_n$ form a basis of the cotangent space at every point in $L^n$. In particular all $\beta_i$ are independent at all points of $L^n$.
\end{proof}

We now prove Liouville theorem {using Corollary \ref{thm:main}}.
\begin{theorem}[Liouville's theorem] The compact regular fiber of an integrable system on $(M^{2n}, \omega)$ is diffeomorphic to a torus $\mathbb T^n$.
\end{theorem}

\begin{proof} The forms $\beta_i$ in the preceding lemma are $n$ closed $1$-forms which are  independent  at all points of $L^n$. Applying theorem \ref{thm:main}, $L^n\cong \mathbb T^n$.
\end{proof}

\subsection{ Commutative and non-commutative integrable systems on Poisson manifolds}
\hfill \newline
A Poisson manifold is a pair $(M, \Pi)$ where $\Pi$ is a bi-vector field with an associated bracket on functions
$$\{f,g\}:=\Pi(df,dg), \quad f,g:M\to \R$$
satisfying the Jacobi identity. This is equivalent to the integrability equation $[\Pi, \Pi]=0$.
The Hamiltonian vector field of a function $f$ in this context  is defined as
$X_f := \Pi(df,\cdot)$.

Poisson manifolds constitute a generalization of symplectic manifolds and it generalizes very natural structures such that of linear Poisson structures associated to the dual of a Lie algebra. Integrability of Hamiltonian systems in the Poisson setting is a rich field which is naturally connected to representation theory (Gelfand-Ceitlin systems on {$\mathfrak{u}(n)^*$)}.
We recall the notion of integrable system for Poisson manifolds.

\begin{defi}\label{def:liouville}
  Let $(M,\Pi)$ be a Poisson manifold of (maximal) rank $2r$ and of dimension $n$. An $s$-tuplet of functions
  $\F=(f_1,\dots,f_s)$ on $M$ is said to define a \emph{Liouville integrable system} on $(M,\Pi)$ if
  \begin{enumerate}
    \item $f_1,\dots,f_s$ are independent (i.e., their differentials are independent on a dense open set),
    \item $f_1,\dots,f_s$ are pairwise in involution and $r+s=n$.
  \end{enumerate}
The map $\F:M\to\R^s$ is called the \emph{moment map} of $(M,\Pi,\F)$.
\end{defi}

{ A point $m$ is called regular whenever $\diff_mf_1\we\dots\we\diff_mf_s\neq0$. Observe that complete integrability in the Poisson context also implies that the distribution generated by ${X_{f_1},\dots,X_{f_s}}$ is integrable in a neighborhood of a regular point in the sense of Frobenius because $[X_{f_i}, X_{f_j}]= X_{\{f_i, f_j\}}$.}

{ For integrable systems on Poisson manifolds it is possible to prove that the compact leaves of this distribution are tori and indeed to prove an action-angle theorem in a neighborhood of a regular torus {(see \cite{LMV})}.}

\begin{theorem}\label{thm:action-anglepoisson}
  Let $(M,\Pi)$ be a Poisson manifold of dimension $n$ of maximal rank $2r$. Suppose that $\F=(f_1,\dots,f_s)$
  is an integrable system on $(M,\Pi)$, i.e., $r+s=n$ and the components of $\F$ are independent and in
  involution. Suppose that $m\in M$ is a point such that
\begin{enumerate}
  \item[(1)] $\diff_mf_1\we\dots\we\diff_mf_s\neq0$;
  \item[(2)] The rank of $\Pi$ at $m$ is $2r$;
  \item[(3)] The integral manifold $\FF_m$ of the { distribution generated by} ${X_{f_1},\dots,X_{f_s}}$, passing through $m$, is compact.
\end{enumerate}
  Then there exists ${\R} $-valued smooth functions $(\sigma_1,\dots, \sigma_{s})$ and $ {\R}/{\Z}$-valued smooth
  functions $({\theta_1},\dots,{\theta_r})$, defined in a neighborhood $U$ of $\FF_m$ such that
  \begin{enumerate}

  \item The manifold $\FF_m$ is a torus $\T^r$.

    \item The functions $(\theta_1,\dots,\theta_r,\sigma_1,\dots,\sigma_{s})$ define an isomorphism
      $U\simeq\T^r\times B^{s}$;
    \item The Poisson structure can be written in terms of these coordinates as
      $$
        \Pi=\sum_{i=1}^r\pp{}{\theta_i}\we\pp{}{\sigma_i},
      $$
      in particular the functions $\sigma_{r+1},\dots,\sigma_{s}$ are Casimirs of $\Pi$ (restricted to $U$);
    \item The leaves of the surjective submersion $\F=(f_1,\dots,f_{s})$ are given by the projection onto the
          second component $\T^r\times B^{s}$, in particular, the functions $\sigma_1,\dots,\sigma_{s}$ depend on the
          functions $f_1,\dots,f_{s}$ only.
  \end{enumerate}

\end{theorem}

The functions $\theta_1,\dots,\theta_r$ are called \textbf{angle coordinates}, the functions $\sigma_1,$
$\dots,\sigma_r$ are called \textbf{action coordinates} and the remaining functions
$\sigma_{r+1},\dots,\sigma_{s}$ are called \textbf{transverse coordinates}.

We can apply Tichler's trick to reprove the orbits of a non-commutative integrable systems are tori in the Poisson setting (and deduce the result for the commutative particular case). We start recalling some definitions from \cite{LMV}.

\begin{defi}\label{def:noncomintsys}
 Let $(M,\Pi)$ be a Poisson manifold. An $s$-tuple of functions  $F=(f_1,\dots,f_s)$ on $M$ is a \textbf{non-commutative (Liouville) integrable system} of rank {$r\leq s$} on $(M,\Pi)$ if
\begin{enumerate}
  \item $f_1,\dots,f_s$ are independent (i.e. their differentials are independent on a dense open subset of $M$) and the Hamiltonian vector fields of the functions $f_1, \dots,f_r$ are linearly independent at some point of $M$.
  \item The functions $f_1,\dots,f_r$ are in involution with the functions $f_1,\dots,f_s$ and $r+s =\dim M$.

\end{enumerate}
\end{defi}
\begin{remark}
As a consequence the maximal rank of the Poisson structure is $2r$.
\end{remark}

\begin{remark} When $r=s$ and thus all the first integrals commute we obtain Liouville integrable systems as a particular case.
\end{remark}

Some notation: We denote the subset of $M$ where the differentials $d f_1,\dots,d f_s$ are independent by  $U_F$  and the subset of $M$ where the vector fields
 $X_{f_1}, \dots ,X_{f_r}$ are independent by  $M_{F,r}$.

On the open subset $M_{F,r}\cap U_F $ of~$M$, the Hamiltonian vector fields
$X_{f_1},\dots, X_{f_r}$ define an involutive distribution of rank $r$. Let us denote by  ${\mathcal F} $ its foliation with $r$-dimensional leaves, see \cite{LMV}.
{When $\FF_m$ is a compact $r$-dimensional manifold,  the action-angle coordinate theorem  proved in  \cite{LMV} (Theorem 1.1) proves that $\FF_m$ is a torus and gives a semilocal description of the Poisson structure in a neighborhood of a compact invariant set:}

\begin{theorem}\label{thm:action-angle_intro}
  Let $(M,\Pi,F)$ be a non-commutative integrable system of rank
  $r$, where $F=(f_1,\dots,f_s)$ and suppose that
  $\FF_m$ is compact, where $m\in M_{F,r}\cap U_F$.  Then there exist ${\R} $-valued smooth
  functions $(p_1,\dots, p_r,z_1, \dots, z_{s-r} )$ and $ {\R}/{\Z}$-valued smooth functions
  $({\theta_1},\dots,{\theta_r})$, defined in a neighborhood $U$ of $\FF_m$, and functions
  $\phi_{kl}=-\phi_{lk}$, which are independent of $\theta_1,\dots,\theta_r,p_1,\dots,p_r$, such that
  \begin{enumerate}

  \item $\FF_m$ is a torus $\T^r$.
    \item The functions $(\theta_1,\dots,\theta_r,p_1,\dots,p_{r},z_1 ,\dots, z_{s-r} )$ define a diffeomorphism
            $U\simeq\T^r\times B^{s}$;
    \item The Poisson structure can be written in terms of these coordinates as,
\begin{equation*}
       \Pi=\sum_{i=1}^r \pp{}{\theta_i}\we\pp{}{p_i} + \sum_{k,l=1}^{s-r} \phi_{kl}(z) \pp{}{z_k}\we\pp{}{z_l} ;
    \end{equation*}
    \item The leaves of the surjective submersion $F=(f_1,\dots,f_{s})$ are given by the projection onto the
      second component $\T^r \times B^{s}$ and as a consequence the functions $f_1,\dots,f_s$ depend on
      $p_1,\dots,p_r,z_1, \dots, z_{s-r}$ only.
  \end{enumerate}
\end{theorem}

Let us now prove the first part of the theorem above using Corollary \ref{thm:main}.

\begin{proof}{(of first item above)}

Consider $F=(f_1,\dots, f_s)$ the set of first integrals of the non-commutative integrable system. Consider the span of the Hamiltonian vector fields $X_i:=\Pi(df_i, \cdot)$. From definition of the non-commutative integrable system at each point on the regular set, dimension of the vector space is $r$.
     Denote by $\alpha_i$ the $1$-forms such that  $\alpha_i(X_j)=\delta_{ij}$. They are not uniquely determined, but if we consider the inclusion $j$ of the orbit into the manifold, then the $1$-forms $\beta_i=j^*\alpha_i$ are uniquely determined. We can easily check that the forms  $\beta_i$ are closed:  \begin{align*}
d\beta_i (X_j,X_k)&=  X_k(\beta_i(X_j)) - X_j(\beta_i(X_k)) - \beta_i([X_j,X_k]) \\
				   &= X_k( \delta_{ij}) - X_j (\delta_{ik}) - \beta_i(0)=0
				   \end{align*}
where in the last equality we have used that $[X_j,X_k]=X_{\{f_j, f_k\}}$ and from the definition of non-commutative integrable system $X_{\{f_j, f_k\}}=X_0=0$.

From the definition the dimension of the orbit is $r$ and we have exactly $r$ forms thus applying  Corollary \ref{thm:main} we conclude that the orbit is a torus. From the regular value theorem, observe also that this orbit is the connected component through the point of the mapping given by
$F=(f_1,\dots,f_{s})$.
\end{proof}

Finally, when $r=s$ we obtain as corollary the first statement of Theorem 3.7 of the commutative case.

\begin{corollary} Given an integrable system on a Poisson manifold $\F=(f_1,\dots,f_s)$, the regular integral manifold $\FF_m$ of { the distribution generated by} ${X_{f_1},\dots,X_{f_s}}$, passing through $m$, is a torus of dimension $r$, $\T^r$.

\end{corollary}

\end{document}